\documentclass[a4paper,twoside]{article}

\newcommand{\heuteIst}{July 10, 2010 }

\usepackage{amsmath,amsthm,amsfonts,latexsym,amscd,amssymb,enumerate}
\usepackage[hypertex]{hyperref}

\usepackage{amsrefs}

\swapnumbers
\theoremstyle{plain}
\newtheorem{theorem}{Theorem}[section]

\theoremstyle{definition}

\newtheorem{example}[theorem]{Example}

\theoremstyle{remark}

\newcommand{\reals}{\mathbb{R}}

\newcommand{\integers}{\mathbb{Z}}
\newcommand{\rationals}{\mathbb{Q}}

\DeclareMathOperator{\id}{id}

\newcommand{\tensor}{\otimes}

\newcommand{\iso}{\cong}

\DeclareMathOperator{\im}{im}      

\DeclareMathOperator{\coker}{coker}

\newcommand{\forget}[1]{}

{\catcode`@=11\global\let\c@equation=\c@theorem}

\allowdisplaybreaks[2]

\begin{document}
\pagestyle{myheadings}
\markboth{Ulrich Bunke and Thomas Schick}{Corrigendum: Uniqueness of smooth extensions of generalized cohomology theories}

\date{Last compiled \today; last edited  \heuteIst or later}

\title{Corrigendum: Uniqueness of smooth extensions of generalized cohomology theories}

\author{Ulrich Bunke\thanks{ulrich.bunke@mathematik.uni-regensburg.de}\\
NWF I - Mathematik,
Universit{\"a}t Regensburg\\
93040 Regensburg, 
Germany   \and Thomas Schick\thanks{
\protect\href{mailto:schick@uni-math.gwdg.de}{e-mail:
  schick@uni-math.gwdg.de}
\protect\\
\protect\href{http://www.uni-math.gwdg.de/schick}{www:~http://www.uni-math.gwdg.de/schick}
\protect\\
Fax: ++49 -551/39 2985\protect\\
partially funded by the Courant Research Center "`Higher order structures in Mathematics"'
within the German initiative of excellence
}\\
Georg-August-Universit\"at G{\"o}ttingen\\
Bunsenstr.~3, 37073 G\"ottingen, 
Germany}
\maketitle

\begin{abstract}
  The proof of Theorem 7.12 of ``Uniqueness of smooth cohomology
  theories'' by the authors of this note is not correct. The said theorem
  identifies the flat part of a 
  differential 
  extension of a generalized cohomology theory $E$ with
  $E\reals/\integers$ (there called ``smooth extension''). In this note, we
  give a correct proof. Moreover, we 
  prove slightly stronger versions of some of the other results of that paper.
\end{abstract}

\section{Introduction}

Given a generalized cohomology theory $E$, the paper \cite{BSuniqueness}
studies its differential extensions $\hat E$ (called there "`smooth"'
extension). These are characterized by a natural list of axioms \cite[Section 1]{BSuniqueness}. The main
question answered is: under which assumptions is $\hat E$ already determined
by $E$ and the axioms. 

One of the axioms requires a natural transformation $R\colon
\hat E(X)\to \Omega(X; E(pt)\tensor\reals)$. Its kernel $\ker(R)=: \hat
E_{flat}$ is automatically a homotopy
invariant functor, and a second goal of \cite{BSuniqueness} is to study
this functor, and in particular to identify conditions under which it
is isomorphic to $E\reals/\integers$, ``$E$ with coefficients in
$\reals/\integers$''. The main result of \cite{BSuniqueness} about this is
\cite[Theorem 7.12]{BSuniqueness}. It turns out that, although the statement
of this theorem essentially is correct, the proof given there is not. The
purpose of this 
corrigendum is to provide a correct proof.

 Along the way, we observe several
small generalizations of some of the results obtained in the paper.
Differential extensions naturally come in two flavors: there are versions which
are defined on all smooth manifolds, and versions defined only for the
category of compact smooth manifolds (but possibly with boundary). Indeed,
important examples exist which are only constructed for compact manifolds. 
Rather strong assumptions on $E$ are required to prove the uniqueness results
of \cite{BSuniqueness}. In the present note we remark that this can be relaxed
slightly. 

\section{Generalizations of the uniqueness results}

All the uniqueness results of \cite{BSuniqueness} require that the generalized
cohomology theory $E$ is rationally even, i.e.~$E^{2k+1}(pt)\tensor \rationals
=\{0\}$ for all $k\in\integers$ and that $E^*(pt)$ is countably
generated as abelian group. Additionally, if the differential extension is only
defined on the category 
of compact manifolds, we require that $E^k(pt)$ is finitely generated as
abelian group for all $k\in\integers$ and that $E^{2k+1}(pt)=\{0\}$. 

\begin{theorem}\label{theo:gen_uniq}
  The uniqueness statements of \cite{BSuniqueness} for differential extensions
  defined on the category of compact manifolds also hold if we drop the
  additional assumption that $E^{2k+1}(pt)=\{0\}$. Note that we still require
  that these groups are finitely generated abelian torsion group, i.e.~finite
  abelian groups.

  This applies in particular to \cite[Theorem 3.10, Corollary 4.4, Theorem
  5.5]{BSuniqueness}. 
\end{theorem}

\begin{example}
  The uniqueness results of \cite{BSuniqueness} also hold for oriented
  bordism, unoriented bordism, spin bordism or real K-theory. 
\end{example}
\begin{proof}
  All these cohomology theories are degree-wise finitely generated and
  rationally even. Note, however, that they do not satisfy the previously made
  conditions in \cite{BSuniqueness}.
\end{proof}

\begin{proof}[Proof of Theorem \ref{theo:gen_uniq}.]
  The proof of the uniqueness results of \cite{BSuniqueness} is based on the
  approximation of classifying spaces for $E^{2k}$ by manifolds to which the
  differential extension can be applied. The important technical tool is
  \cite[Proposition 2.1]{BSuniqueness}. Along the way of the proof, we
  construct finite connected CW-complexes $X$ and we have to know that their
  higher homotopy 
  groups are finitely generated abelian groups. This is known to be the case
  if $X$ is simply connected, which follows from the assumption that our
  $E^{2k+1}(pt)=0$ for all $k\in\integers$. However, the required finite
  generation also follows if $X$ has a finite fundamental group, by applying
  the old argument to the universal covering, which in this case is still a
  finite CW-complex. Correspondingly, all proofs work exactly the same way if
  $E^{2k+1}(pt)$ is finite for all $k\in\integers$.
\end{proof}

\section{Correction of \cite[Theorem 7.12]{BSuniqueness}}
  
Given a differential extension $(\hat E,\int)$ of any generalized cohomology
theory, which comes with the additional structure of a transformation
``integration over an $S^1$-fiber'', we show in \cite[Theorem
7.11]{BSuniqueness} that $\hat E_{flat}$, the kernel of the curvature
transformation, is a homotopy invariant functor with Mayer-Vietoris
sequences, defined either on the category of compact manifolds, or on the
category of all manifolds. On can then extend this uniquely to a generalized
cohomology theory, either defined on all finite CW-complexes, or defined on
all finite dimensional countable CW-complexes. 

Moreover, we prove that 
the transformation $a$ restricts to a natural transformation $a\colon
E\reals^{*-1}(X)\to \hat E_{flat}^*(X)$ and (by construction, although this is not
stated explicitly in \cite{BSuniqueness}) the map $I$ restricts to a natural
transformation $I\colon \hat E_{flat}^*(X)\to E^*(X)$ of cohomology theories
defined on finite CW-complexes. 

Moreover, for each finite CW-complex $X$ we get a long exact sequence
\begin{equation*}
  \to E\reals^{*-1}(X)\xrightarrow{a} \hat E_{flat}^*(X)\xrightarrow{I}
  E^*(X)\xrightarrow{ch}   E\reals^*(X)\to 
\end{equation*}

In other words, the cohomology theory on finite CW-complexes $\hat E_{flat}^*$
fits into the same exact sequence as $E\reals/\integers^{*-1}$. The goal of
\cite[Theorem 7.12]{BSuniqueness} now is to compare these two cohomology
theories. We can realize $\hat E_{flat}$ by a spectrum $U$ (which is unique
up to equivalence) and the transformation $a$ (as well as $I$) by a map of
spectra $a\colon E\reals\to U$ (which are in general not uniquely determined by the transformation of
cohomology theories defined on the category of finite CW-complexes only). Let
the spectrum $F$ be the fiber of $a\colon E\to U$.

\begin{theorem}
  Let $E$ be a generalized cohomology theory and $(\hat E,\int)$ a
  differential 
  extension with integration as above. If $E^k$ is finitely generated for each
  $k$ then we
  obtain a natural 
  transformation of cohomology theories $\phi\colon \hat E_{flat}^*\to
  E\reals/\integers^{*-1}$ (and an associated transformation $\phi_F\colon
  F\to E$) such that one has a commutative diagram
  \begin{equation}\label{eq:general}
    \begin{CD}
      @>>> F^*(X) @>>> E\reals^*(X) @>a>> \hat E_{flat}^{*+1}(X) @>>> \\
      && @VV{\phi_F}V @VV=V @VV{\phi}V\\
      @>>> E^*(X) @>>> E\reals^*(X) @>>> E\reals/\integers^*(X) @>>>
    \end{CD}
  \end{equation}
  For each finite CW-complex $X$, we know that $\ker(a)=\im(E^*(X))$ and
  $\coker(a)\iso E^{*+1}_{tors}(X)$, where $E^*_{tors}(X)$ is the torsion
  subgroup of $E^*(X)$. Therefore, the above diagram yields the diagram of
  exact sequences
  \begin{equation}\label{eq:special}
    \begin{CD}
      E^*(X) @>>> E\reals^*(X) @>a>> \hat E^{*+1}_{flat}(X) @>>>
      E^{*+1}_{tors}(X)\to 0\\
      @VV=V @VV=V @VV{\phi}V \\
      E^*(X) @>>> E\reals^*(X) @>>> E\reals/\integers^*(X) @>>>
      E^{*+1}_{tors}(X) \to 0.
    \end{CD}
  \end{equation}
  Note that we do not claim (and don't know) whether the diagram can be
  completed to a commutative diagram by $\id\colon E_{tors}^{*+1}(X)\to
  E_{tors}^{*+1}(X)$. 
  If $E^*_{tors}(pt)=0$ then $\phi$ is an isomorphism of cohomology theories.
\end{theorem}
\begin{proof}
  We know by the axioms that for each compact manifold $X$ the kernel of
  $a\colon E\reals^*(X) \to \hat E_{flat}^{*+1}(X)$ is exactly the image of
  the canonical map $ch\colon E^*(X)\to E\reals^*(X)$. However, because of the
  long exact sequence associated to the fiber sequence $F\to E\xrightarrow{a}
  U$ of spectra, we conclude that the image of $F(X)$ in $E\reals(X)$
  coincides with the image of $E(X)$. This means by definition that the
  composed map of spectra  $F\to E\reals\to E\reals/\integers$ is a phantom
  map. However, under the assumption the $E^k$ is finitely generated for each
  $k$, it is shown in \cite[Section 8]{BSuniqueness} that such a phantom map
  is automatically trivial (absence of phantoms). The existence of the
  required factorizations $\phi$, $\phi_F$ is now a direct consequence for the
  triangulated category of spectra and the fiber sequences (distinguished
  triangles) $F\to E\reals\to U$, $E\to E\reals\to E\reals/\integers$. By
  definition of factorization, we obtain the commutative diagram of exact
  sequences 
  \eqref{eq:general}, which because of the knowledge about image and kernel of
  $a$ specializes to \eqref{eq:special}.

  If $E_{tors}^*(pt)=0$, the five lemma implies that $\phi\colon U\to
  E\reals/\integers$ induces an isomorphism on the point and therefore for all
  finite CW-complexes.
\end{proof}


\end{document}